\documentclass[11pt,twoside]{article}
\usepackage{a4wide,color,bm,amsmath}
\usepackage{amssymb}
\usepackage{amsmath}

\newtheorem{theorem}{Theorem}

\newtheorem{corollary}{Corollary}
\newtheorem{lemma}{Lemma}

\renewcommand{\Bbb}[1]{\mathbb{#1}}
\newcommand{\N}{{\Bbb N}}         
\newcommand{\Q}{{\Bbb Q}}         
\newcommand{\R}{{\Bbb R}}         
\newcommand{\Z}{{\Bbb Z}}         

\newcommand{\La}{\Lambda}

\newcommand{\cD}{{\cal D}}
\newcommand{\cE}{{\cal E}}

\newcommand{\cM}{{\cal M}}
\newcommand{\cN}{{\cal N}}

\newcommand{\cS}{{\cal S}}

\newcommand{\ve}{\varepsilon}

\newcommand{\vv}[1]{{\mathbf{#1}}}

\renewcommand{\le}{\leqslant}
\renewcommand{\ge}{\geqslant}
\newcommand{\nz}{\smallsetminus\{0\}}

\newenvironment{proof}{\textsc{Proof.}}{\ \newline\hspace*{\fill}$\boxtimes$}

\newcommand{\SIE}{\text{\rm SIE}}
\newcommand{\DIE}{\text{\rm DIE}}

\newcommand{\qqand}{\qquad\text{and}\qquad}

\begin{document}

\title{Simultaneous inhomogeneous Diophantine \\ approximation on manifolds}

\author{Victor Beresnevich\footnote{EPSRC Advanced Research Fellow, EP/C54076X/1}
\\ {\small\sc York} \and Sanju Velani\\ {\small\sc York}}

\date{{\it Dedicated to A.O.~Gelfond on what would have \\ been his 100th birthday}}

\maketitle

\begin{abstract}
In 1998, Kleinbock \& Margulis \cite{KleinbockMargulis-1998}
established a conjecture of V.G.~Sprindzuk in metrical Diophantine
approximation (and indeed the stronger Baker-Sprindzuk
conjecture). In essence the conjecture stated that the
simultaneous homogeneous Diophantine exponent $w_{0}(\vv x) = 1/n$
for almost every point $\vv x$ on a non-degenerate submanifold
$\cM$ of $\R^n$. In this paper the simultaneous inhomogeneous
analogue of Sprindzuk's conjecture is established. More precisely,
for any `inhomogeneous' vector $\bm\theta\in\R^n$  we prove that
the simultaneous inhomogeneous Diophantine exponent $w_{0}(\vv x,
\bm\theta)= 1/n$ for almost every point $\vv x$ on $M$.
The key result is an {\em inhomogeneous transference principle}
which enables us to deduce that the homogeneous exponent  $w_0(\vv
x)=1/n$ for almost all $\vv x\in \cM$ if and only if for any
$\bm\theta\in\R^n$ the inhomogeneous exponent $w_0(\vv
x,\bm\theta)=1/n$ for almost all $\vv x\in \cM$. The inhomogeneous
transference principle introduced in this paper is an extremely
simplified version of that recently discovered in
\cite{Beresnevich-Velani-new-inhom}. Nevertheless, it should be emphasised  that
 the simplified version has the great advantage of
bringing to the forefront the main ideas of
\cite{Beresnevich-Velani-new-inhom} while omitting the abstract and
technical notions that come with describing the inhomogeneous
transference principle in all its glory.
\end{abstract}


\section{Introduction}

The metrical theory of Diophantine approximation on manifolds
dates back to 1932 with a conjecture of K.~Mahler
\cite{Mahler-1932b} in transcendence theory. The conjecture was
easily seen to be equivalent to a metrical Diophantine
approximation problem restricted to Veronese curves. Mahler's
problem remained a key open problem in metric number theory for
over 30 years and was eventually solved by Sprindzuk
\cite{Sprindzuk69} in 1964. Moreover, its solution eventually lead
Sprindzuk \cite{Sprindzuk80} to make an important general
conjecture which we now describe. For a vector $\vv x\in\R^n$, let
$$
w_0(\vv x):=\sup\{w:\|q\vv x\|<q^{-w}\text{  for i.m. }q\in\N\} \ ,
$$
and
$$
w_{n-1}(\vv x):= \sup\{w:\|\vv q.\vv x\|<|\vv q|^{-w}\text{
 for i.m. }\vv q\in\Z^n\nz\}\,.
$$
Here and elsewhere `i.m.' is the abbreviation for `infinitely
many', $|\vv q|:= \max\{|q_1|, \ldots, |q_n|\} $ is the supremum
norm, $\vv q.\vv x:= q_1 x_1 + \ldots + q_n x_n $ is the standard
inner product and $\|\cdot\|$ is the distance to the nearest
integer. For obvious reasons, $w_0( \vv x )$ is referred to as the
{\em simultaneous Diophantine exponent} and $w_{n-1}(\vv x)$ is
referred to as  the {\em dual Diophantine exponent}.

A trivial consequence of Dirichlet's theorem, or simply the
`pigeon-hole principle', is  that
\begin{equation}\label{e:001}
 w_{0}(\vv x)\ge\frac1n\qquad\text{and}\qquad w_{n-1}(\vv x)\ge
 n\qquad\text{for all }\vv x\in\R^n\,.
\end{equation}
The Diophantine exponents $w_0(\vv x)$ and $w_{n-1}(\vv x)$ can in
principle be infinite. For example, this is the case when $n=1$
and $x$ is a Liouville number. Nevertheless,  a relatively easy
consequence of the Borel-Cantelli lemma in probability theory is
that the inequalities in (\ref{e:001}) are reversed for almost all
$\vv x\in\R^n$ with respect to Lebesgue measure on $\R^n$. Thus,
\begin{equation*}
 w_{0}(\vv x)=\frac1n\qquad\text{and}\qquad w_{n-1}(\vv x)=
 n\qquad\text{for almost all }\vv x\in\R^n\,.
\end{equation*}
Sprindzuk conjectured that a similar statement holds for any
non-degenerate submanifold $\cM$ in $\R^n$ with respect to the
Lebesgue measure induced on $\cM$. Essentially, these are smooth
submanifolds of $\R^n$ which are sufficiently curved so that they
deviate from any hyperplane with a `power law' (see
\cite{Beresnevich-2002a}). Formally, a differentiable manifold
$\cM$ of dimension $d$ embedded in $\R^n$ is said to be
non-degenerate if there exists an atlas $\{\cM_i,\vv
g_i\}_{i\in\N}$ such that each map $\vv g_i:\cM_i\to U_i$, where
$U_i$ is an open subset of $\R^d$, is a diffeomorphism and $\vv
f_i:=\vv g^{-1}_i$ is non-degenerate. The map $\vv f:U\to \R^n:\vv
u\mapsto \vv f(\vv u)=(f_1(\vv u),\dots,f_n(\vv u))$ is said to be
non--degenerate at $\vv u\in U$ if there exists some $l\in\N$ such
that $\vv f$  is $l$ times continuously differentiable on some
sufficiently small ball centred at $\vv u$ and the partial
derivatives of $\vv f$ at $\vv u$ of orders up to $l$ span $\R^n$.
The map $\vv f$ is non--degenerate if it is non--degenerate at
almost every (in terms of $d$--dimensional Lebesgue measure) point
in $U$.  Any real, connected analytic manifold not contained in
any hyperplane of $\R^n$  is easily seen to be non--degenerate.
For a planar curve, non-degeneracy is simply equivalent to the
condition that the curvature is non-vanishing almost everywhere.
In short, non-degeneracy is a natural generalisation of non-zero
curvature and naturally excludes obvious counterexamples to
Sprindzuk's conjecture which we now formally state.

\bigskip

\noindent\textbf{Sprindzuk's conjecture. \ } {\it Let $\cM$ be a
non-degenerate submanifold of $\R^n$. Then
\begin{equation}\label{e:003}
 w_{0}(\vv x)=\frac1n\qquad\text{and}\qquad w_{n-1}(\vv x)=
 n\qquad\text{for almost all }\vv x\in \cM\,.
\end{equation}
}

\bigskip

\noindent In the case $\cM:=\{(x,x^2,\dots,x^n):x\in\R\}$,
Sprindzuk's conjecture reduces to  Mahler's problem.  Manifolds
that satisfy (\ref{e:003}) are simply refereed to as
\emph{extremal}. Thus, Sprindzuk's conjecture simply states that
any non-degenerate submanifold of $\R^n$ is extremal. To be
absolutely precise,  the definition of non-degeneracy given in the
form  above was actually  introduced by Kleinbock \& Margulis and
not Sprindzuk.  Essentially, Sprindzuk   considered the case of
analytic manifolds.

\medskip


\noindent {\em Remark. }It is worth stressing that the equalities in
(\ref{e:003}) concerning the simultaneous Diophantine exponent
$w_{0}(\vv x)$ and the dual Diophantine exponent $w_{n-1}(\vv x)$
are intimately related via a classical transference principle -- see
\S\ref{TI}. Thus, in order to establish the above conjecture it
suffices to prove either of the two equalities. In other words, the
notion of extremal is actually independent of the type of
Diophantine exponent under consideration and we may freely work
within either the simultaneous framework or the dual framework
depending on what is most convenient. A priori, this is not the case
when considering the inhomogeneous analogue of Sprindzuk's
conjecture.

\medskip

By the time Sprindzuk made his conjecture,  the case $n=2$ (planar
curves) had already been established by W.M. Schmidt
\cite{Schmidt-1964b} many years earlier. Until 1998, the numerous
contributions towards  Sprindzuk's conjecture had all been limited
to special classes of manifolds -- the most significant being that
of Beresnevich \& Bernik \cite{Beresnevichbernik2} who proved the
conjecture in the case of curves in $\R^3$.  In ground breaking
work,  Kleinbock \& Margulis \cite{KleinbockMargulis-1998}
established Sprindzuk's conjecture in full generality. Furthermore,
they answered a more general question of A. Baker concerning a
stronger notion of extremality. This is nowadays referred to as the
Baker-Sprindzuk conjecture and will not be discussed within this
paper however the inhomogeneous version of the
Baker-Sprindzuk conjecture is addressed in
\cite{Beresnevich-Velani-new-inhom}.

\bigskip

\noindent\textbf{Theorem A (Kleinbock \& Margulis)} {\it Let $\cM$
be a non-degenerate submanifold of $\R^n$. Then $\cM$ is
extremal.}

\bigskip

Kleinbock \& Margulis used ideas drawn from dynamical systems to
prove their theorem, in particular the theory of flows on
homogeneous spaces.
Independently of their work, Beresnevich \cite{Beresnevich-2002a}
used classical methods to establish a convergence
(Khintchine-Groshev type) criterion from which Theorem A readily
follows. Recently,  Kleinbock \cite{Kleinbock-2003-extremal} has
extended Theorem A to incorporate non-degenerate submanifolds of
linear subspaces of $\R^n$.  This naturally broadens the class of
extremal manifolds beyond the notion of non-degeneracy. For example,
affine subspaces of $\R^n$ are degenerate and so the result that
non-degenerate manifolds are extremal is not applicable.

\bigskip

\noindent\textbf{Theorem B (Kleinbock)} {\it Let $L$ be an affine
subspace of $\R^n$.

$(a)$ If $L$ is extremal and $\cM$ is a non-degenerate submanifold
of $L$, then $\cM$ is extremal.

$(b)$ If $L$ is not extremal, then no subset of $L$ is extremal. }

\medskip

\noindent Since $\R^n$ is itself extremal, Theorem A is obviously
covered by part (a) of Theorem B.

The main substance of the present work is to establish the
inhomogeneous analogue of Sprindzuk's conjecture in the case of
simultaneous approximation. Naturally, we begin by introducing the
simultaneous inhomogeneous Diophantine exponent. For
$\bm\theta
\in\R^n$, let
$$
w_0(\vv x,\bm\theta):=\sup\big\{w\ge0:\|q\vv
x+\bm\theta\|<|q|^{-w}\text{
 for i.m. }q\in\Z\nz\big\}\,.
$$
A manifold $\cM$ is said to be \emph{simultaneously
inhomogeneously extremal}\/ ($\SIE$ for short) if for every
$\bm\theta \in\R^n$,
\begin{equation}\label{e:004}
 w_0(\vv x,\bm\theta)=\frac1n\qquad\text{for almost all }\vv x\in \cM\,.
\end{equation}

The main result of this paper is the following transference
statement.

\begin{theorem}\label{t1}
Let $\cM$ be a differentiable submanifold of $\R^n$. Then
$$
\cM\text{ is extremal }\iff \cM\text{ is \ SIE}.
$$
\end{theorem}

\noindent There is clearly a trivial part of Theorem~\ref{t1} as any
simultaneously inhomogeneously extremal manifold is  extremal. This
is simply due to the fact that  $w_0(\vv x)=w_0(\vv x,\vv0)$. Thus
the converse part of Theorem 1 constitutes the main substance.
Indeed, it is rather surprising that a homogeneous statement ($\cM$
is  extremal)  implies an inhomogeneous
statement ($\cM$ is  SIE). 
The philosophy behind the above inhomogeneous transference
principle is broadly comparable with  the recently  discovered
mass transference principle \cite{MTP,Slicing} in metric number
theory.

The following results are simple consequences of Theorem \ref{t1}
and represent the simultaneous inhomogeneous analogues of Theorems
A and B.

\bigskip

\noindent\textbf{Corollary A } {\it Let $\cM$ be a non-degenerate
manifold  of $\R^n$. Then $\cM$  is SIE.}

\bigskip

\noindent\textbf{Corollary B } {\it Let $L$ be an extremal affine
subspace of $\R^n$ and let $\cM$ be a non-degenerate submanifold of
$L$. Then both $L$ and $\cM$ are SIE.}

\bigskip

We stress that Corollary A alone does not establish the complete
inhomogeneous analogue of Sprindzuk's conjecture. For this we would
also have to establish the analogue of Corollary A for the dual form
of approximation.  More precisely, for $\bm\theta \in\R^n$ let
$$
w_{n-1}(\vv x,\bm\theta):= \sup\{w>0:\|\vv q.\vv x +\bm\theta
\|<|\vv q|^{-w}\text{ for i.m. }\vv q\in\Z^n\nz\}\,.
$$
A manifold $\cM$ is said to be \emph{dually inhomogeneously
extremal}\/ ($\DIE$ for short) if for every $\bm\theta \in\R^n$,
\begin{equation*}
 w_{n-1}(\vv x,\bm\theta)=n \qquad\text{for almost all }\vv x\in \cM\, .
\end{equation*}
Moreover, a manifold $\cM$ is simply said to be
\emph{inhomogeneously extremal} if it is both
$\SIE$ and $\DIE$. The following statement represents the natural
inhomogeneous analogue of Sprindzuk's conjecture.

\bigskip

\noindent\textbf{Conjecture IE } {\it Let $\cM$ be a non-degenerate
submanifold of $\R^n$. Then $\cM$ is inhomogeneously extremal. }

\bigskip

The following corollary is a simple consequence of the general
framework developed in \cite{Beresnevich-Velani-new-inhom} and
together with Corollary A establishes the above conjecture.

\bigskip

\noindent\textbf{Corollary A${}^\prime$} {\it Let $\cM$ be a
non-degenerate submanifold of $\R^n$. Then $\cM$ is \DIE.}

\bigskip

Unlike in the homogeneous case, there is no classical transference
principle that allows us to deduce  Corollary A$^\prime$ from
Corollary A and vice versa. The upshot is that the two forms of
inhomogeneous extremality, namely SIE and DIE a priori have to be
treated separately. It turns out that establishing the dual form
of inhomogeneous extremality is technically far more complicated
than establishing the simultaneous form. The framework developed
in \cite{Beresnevich-Velani-new-inhom} naturally incorporates both
forms of inhomogeneous extremality and indeed other stronger
notions associated with the inhomogeneous analogue of the
Baker-Sprindzuk conjecture. In particular, the general
inhomogeneous transference principle of
\cite{Beresnevich-Velani-new-inhom} enables us to establish the
following transference for non-degenerate manifolds:
$$
\cM\text{ is extremal }\iff \cM\text{ is inhomogeneously
extremal}.
$$
This together with Theorem A clearly establishes the inhomogeneous
extremality conjecture and also enables us to deduce that:
$$
\cM\text{ is SIE }\iff \cM\text{ is DIE}.
$$
In other words, a transference principle between the two forms of
inhomogeneous extremality does exist at least for the class of
non-degenerate manifolds.

As indicated above, the inhomogeneous transference principle
introduced in this paper (Theorem \ref{t1}) is an extremely
simplified version of that in \cite{Beresnevich-Velani-new-inhom}.
Nevertheless and most importantly, the simplified version has the
great advantage of bringing to the forefront the main ideas of
\cite{Beresnevich-Velani-new-inhom} and at the same time leads to a
transparent and self contained proof of the inhomogeneous analogue
of Sprindzuk's conjecture in the case of simultaneous approximation
-- Corollary A.

\section{Diophantine exponents and transference inequalities \label{TI} }

Transference inequalities in the theory of Diophantine approximation
are often attributed to Khintchine who established the first set of
inequalities relating the dual and simultaneous Diophantine
exponents. Recently, Bugeaud \& Laurent \cite{Bugeaud-Laurent2005}
have discovered new transference inequalities which we shall
conveniently make use of in our proof of Theorem \ref{t1}. In this
section we give a brief overview of these new and classical
transference results.

We start by recalling the classical transference principle of
Khintchine.  For any $\vv x\in\R^n$, {\em Khintchine's transference
principle} relates the Diophantine exponents $w_0(\vv x)$ and
$w_{n-1}(\vv x)$  in the following way:
\begin{equation*}
\frac{w_{n-1}(\vv x)}{(n-1)w_{n-1}(\vv x)+n}\ \le\ w_0(\vv x)\ \le \
\frac{w_{n-1}(\vv x)-n+1}{n}\,.
\end{equation*}

\noindent These inequalities readily imply that
$$
w_0(\vv x)=\frac1n \ \iff \  w_{n-1}(\vv x)=n\  .
$$
A particular implication for us is that in order to  establish
Sprindzuk's conjecture it suffices to prove either of the above
equalities. Thus, within the homogeneous setting there is no need to
differentiate between the two implicit forms (dual and simultaneous)
of extremality because one naturally implies the other.  As already
mentioned, this is far from the situation within the inhomogeneous
setting. Nevertheless, there are various transference inequalities
between homogeneous and inhomogeneous Diophantine exponents (see
\cite{Cassels-1957}) which we are able to utilise. The recent
inequalities discovered by Bugeaud \& Laurent that relate the above
Diophantine exponents with their uniform counterparts are
particularly relevant to establishing Theorem \ref{t1}.

Uniform Diophantine exponents are defined as follows. Let $\vv
x,\bm\theta\in\R^n$. The {\em simultaneous uniform inhomogeneous
exponent} $\widehat w_0(\vv x,\bm\theta)$ is defined to be the
supremum of real numbers $w$, such that for any sufficiently large
integer $Q$ there exists an integer $q$ so that
$$
\|q\vv x+\bm\theta\|\le Q^{-w}\qqand 0<|q|\le Q\,.
$$
Let $\vv x \in\R^n$ and $\theta\in\R$. The {\em dual uniform
inhomogeneous exponent} $\widehat w_{n-1}(\vv x,\theta)$ is
defined to be the supremum of real numbers $w'$, such that for any
sufficiently large integer $Q$ there exists an integer point $\vv
q$ so that
$$
\|\vv q.\vv x+\theta\|\le Q^{-w'}\qqand 0<|\vv q|\le Q\,.
$$
If $\bm\theta=\vv0\bmod1$, 
the above exponents are naturally referred to as the homogeneous
uniform Diophantine exponents and we write  $\widehat w_0(\vv x)$
for $ \widehat w_0(\vv x,\vv0)$ and $\widehat w_{n-1}(\vv x)$ for
$ \widehat w_{n-1}(\vv x,0)$.

 A trivial consequence of
Dirichlet's theorem is that
\begin{equation}\label{e:006}
 \widehat w_{0}(\vv x)\ge\frac1n\qquad\text{and}\qquad \widehat w_{n-1}(\vv x)\ge
 n\qquad\text{for all }\vv x\in\R^n\,.
\end{equation}

\noindent  Also it is  easy to see that
$$
 w_0(\vv x,\bm\theta) \ \ge \ \widehat w_0(\vv x,\bm\theta)\ \ge \ 0 \
 \,,
$$

\begin{equation}\label{e:008}
 w_{n-1}(\vv x,\theta) \ \ge \ \widehat w_{n-1}(\vv x,\bm\theta)\ \ge \ 0 \ \,.
\end{equation}

\noindent In follows from the above three inequalities, that the
homogeneous Diophantine exponents are bounded below from 0 for all
$\vv x\in\R^n$. This is not at all the case for the inhomogeneous
exponents. Indeed, if $x\in \R$ is a Liouville number then
$\widehat w_0(x,\theta)$ vanishes for almost all $\theta \in \R$
-- see \cite{Bugeaud-Laurent2005} or \cite{Cassels-1957} for
further details. A
 simplified version of the main result in
\cite{Bugeaud-Laurent2005} is as follows.

\bigskip

\noindent\textbf{Theorem C (Bugeaud \& Laurent) } {\it Let $\vv x,
\bm\theta \in\R^n$. Then
\begin{equation}\label{e:009}
w_0(\vv x,\bm\theta)\ge\frac1{\widehat w_{n-1}(\vv
x)}\qquad\text{and}\qquad\widehat w_0(\vv
x,\bm\theta)\ge\frac1{w_{n-1}(\vv x)}  \  \ ,
\end{equation}
with equalities in $(\ref{e:009})$ for almost every
$\bm\theta\in\R^n$. }

\bigskip

The following statement is a consequence of Theorem C.

\begin{corollary}\label{cor3}
Let $\cM$ be an extremal differentiable submanifold of $\R^n$. Then for every
$\bm\theta\in\R^n$ we have that
\begin{equation}\label{e:010}
w_0(\vv x,\bm\theta)\ge\frac1n\qquad\text{for almost all }\vv x\in
\cM\,.
\end{equation}
\end{corollary}

\begin{proof}
By definition, for an extremal manifold $\cM$ we have that
$w_{n-1}(\vv x)=n$ for almost all $\vv x\in \cM$. This together with
(\ref{e:006}) and (\ref{e:008}), implies  that $\widehat w_{n-1}(\vv
x)=n$ for almost all $\vv x\in \cM$. Thus, for every
$\bm\theta\in\R^n$
\begin{equation*}\label{e:011}
 w_0(\vv x,\bm\theta)\ \stackrel{(\ref{e:009})}{\ge} \
\frac1{\widehat w_{n-1}(\vv x)} \ = \ \frac1n \qquad\text{for
almost all }\vv x\in \cM\,.
\end{equation*}
\end{proof}

\medskip

In view of Corollary~\ref{cor3} and the fact that $w_0(\vv
x)=w_0(\vv x,\vv0)$, the proof of Theorem~\ref{t1} is reduced to
establishing the following statement.

\bigskip

\noindent\textbf{Theorem \ref{t1}\!*} {\it Let $\cM$ be a
differentiable submanifold of $\R^n$. If $\cM$ is extremal, then
for every $\bm\theta\in\R^n$ we have that
\begin{equation*}\label{e:012}
    w_0(\vv x,\bm\theta)\le\frac1n\qquad\text{for almost all }\vv x\in \cM\,.
\end{equation*}
}

\bigskip

\noindent {\em An important remark. } It is worth pointing out
that Corollary \ref{cor3}, which allows us to reduce Theorem
\ref{t1} to Theorem \ref{t1}\!*, can in fact be established
without appealing to Theorem C. Indeed, a proof can be given which
only makes use of classical transference inequalities; namely
Theorem VI of Chapter 5 in \cite{Cassels-1957}. Thus, the proof of
Theorem \ref{t1} is not actually dependent on the recent
developments regarding transference inequalities. However, given
the existence of Theorem C it would be rather absurd to make no
use of it. Furthermore, Theorem C actually gives us information
beyond inequality (\ref{e:010}) that supports our main result
(Theorem~\ref{t1}). More precisely, it enables us to deduce that
inequality (\ref{e:010}) is in fact an equality for almost all
$\bm\theta$. The real significance of Theorem~\ref{t1} is
therefore in establishing (\ref{e:004}) for \emph{all}\/
$\bm\theta$ rather than for \emph{almost all}\/ $\bm\theta$.

\section{Proof of the Theorem \ref{t1}\!* \label{pfmain}}

Throughout, let  $\mu$ denote the induced Lebesgue measure on the
differentiable submanifold $\cM$ in $\R^n$. For $ \ve > 0$, let
$$
 \cS_n^{\bm\theta}(\ve):=\big\{\vv y\in \R^n:\|q\vv y+\bm\theta\|<|q|^{-\frac1n-\ve}\ \mbox{for
 i.m.\ }q\in\Z\nz\big\} \,
$$
and
$$
 \cS_n(\ve):=\cS_n^{\vv0}(\ve)\,.
$$
Theorem $1^*$ will follow on establishing that
\begin{equation}\label{e:013}
\mu\Big(\cS_n^{\bm\theta}(\ve)\cap\cM\Big)=0 \qquad\text{for any }
 \bm\theta\in\R^n  \text{and any } \ve>0 \ ,
\end{equation}
under the hypothesis that $\cM$ is extremal; i.e.
\begin{equation}\label{e:014}
\mu\Big(\cS_n(\ve)\cap\cM\Big)=0 \qquad\text{for  any } \ve>0 \ .
\end{equation}



\subsection{Slicing into  extremal curves}

In this section we show that it is sufficient to prove
Theorem~\ref{t1} and therefore Theorem $1^*$ for extremal
differentiable curves.

Let $\cM$ be an extremal differentiable submanifold of $\R^n$ of
dimension $d>1$. Consider the  local parameterisation of $\cM$
given by
 $$\vv f:U\to\R^n\ :\vv x=(x_1,\dots,x_d)\mapsto\vv f(\vv
x)\in\cM \, $$
where $U$ is a ball in $\R^d$ and $\vv f$ is a
diffeomorphism. Since $\cM$ is extremal and the fact that sets of
full and zero measure are invariant under  diffeomorphisms, the
set
$$\cE:=\{\vv x\in U:w_0(\vv f(\vv x))=1/n\}$$
has full Lebesgue measure in $U$. Now for every $\vv
x'=(x'_2,\dots,x'_d)\in\R^{d-1}$ consider the line $L_{\vv x'}$ in
$\R^d$ given by
$$
L_{\vv x'}\,:=\{\vv x\in\R^d: x_2=x_2',\ \dots,\ x_d=x_d'\}\,.
$$
Also define $\cE_{\vv x'}=\cE\cap L_{\vv x'}$ and $U_{\vv
x'}=U\cap L_{\vv x'}$. Clearly, $U_{\vv x'}$ is either an interval
or is empty and that $\cE_{\vv x'}\subset U_{\vv x'}$. For obvious
reasons, we only consider the case when $U_{\vv
x'}\not=\emptyset$. Since $\cE$ has full measure in $U$, it
follows from Fubini's theorem that for almost every $\vv
x'\in\R^{d-1}$ the set $\cE_{\vv x'}$ has full measure in $U_{\vv
x'}$.  Now let
$\vv f_{\vv x'}$ denote the map   $\vv f$ restricted to $U_{\vv
x'}$. Clearly,  $\vv f_{\vv x'}$ is  a diffeomorphism from $U_{\vv
x'}$ onto $\cM_{\vv x'}=\vv f(U_{\vv x'})$. Since $\cE_{\vv x'}$
has full measure in $U_{\vv x'}$ and $\vv f_{\vv x'}$ is a
diffeomorphism, $\cM_{\vv x'}$ is extremal for almost all $\vv
x'\in\R^{d-1}$. It follows, under the assumption that
Theorem~\ref{t1} is true for curves, that $\cM_{\vv x'}$ is SIE;
i.e. for every fixed $\bm\theta \in\R^n$ the set
$\cE^{\bm\theta}:=\{\vv x\in U:w_0(\vv f(\vv x),\bm\theta)=1/n\}$
has full Lebesgue measure in $U_{\vv x'}$ for almost all $\vv
x'\in\R^{d-1}$. On applying  Fubini's theorem,  we conclude that
$\cE^{\bm\theta}$ has full Lebesgue measure in $U$. Since the
latter holds for every $\bm\theta \in\R^n$, we have established
that $\cM$ is SIE. The upshot of this is that we only need to
establish Theorem $1^*$ in the case that $\cM$ is an  extremal
differentiable curve.

\bigskip

From this point onwards, $\cM$ is an  extremal differentiable
curve in $\R^n$ and $\bm\theta\in\R^n$ is fixed. Let $\vv
f=(f_1,\dots,f_n):I\to\R^n$ be a diffeomorphic parameterisation of
$\cM$, where $I$ is a finite interval and $\vv f(I)\subset \cM$.
Note the fact that  $\vv f(I)$ is not  necessarily  the whole of
$\cM$ is not an issue since  establishing  Theorem $1^*$ for every
patch of $\cM$ suffices. Since $\vv f$ is a diffeomorphism, the
Implicit Function Theorem enables us to change variables  so that
\begin{equation}\label{e:015}
f_1(x)=x.
\end{equation}
This is the standard  \emph{Monge parameterisation}. Also, since
we are able to work locally (i.e. on patches of $\cM$), we can
assume that
\begin{equation}\label{e:016}
    C:=\sup_{x\in I}|\vv f'(x)|<\infty\, ,
\end{equation}
as otherwise we can restrict $\vv f$ to an interval $J$ such that the closure of $J$ is contained
in  $I$. For the same reason, there is no loss of
generality in  assuming that the curve  $\cM$ is bounded.

\subsection{Auxiliary lemmas}

Given $\vv p\in\Z^n$, $q\in\Z\nz$, $\bm\theta\in\R^n$ and $\ve>0$
define the  ball

$$
 B^{\bm\theta}_{\vv p,q}(\ve):=\{\, \vv y\in \R^n:|q \vv y+\vv p+\bm\theta|<|q|^{-\frac1n-\ve} \, \}
\ .
$$

\noindent In general, $B=B(\vv x,r)$ is the ball centred at $\vv
x$ and of radius $r>0$. For any $\lambda >0$,   we denote by
$\lambda B$ the ball $B$ scaled by a factor $\lambda$; i.e.
$\lambda B := B(\vv x,\lambda r)$. In this notation,
$$
B^{\bm\theta}_{\vv p,q}(\ve)= B\Big( (\vv
p+\bm\theta)/q,|q|^{-1-\frac1n-\ve}  \Big) \ .
$$

\begin{lemma}\label{lemma1}
Let $\cM=\{\vv f(x):x\in I\}$ be a curve satisfying
$(\ref{e:015})$ and $(\ref{e:016})$. Then for any choice of $\vv
p\in\Z^n$, $q\in\Z\nz$,  $\ve>0$ and  $\bm\theta\in\R^n$ we have
that
\begin{equation}\label{e:017}
    \mu(B^{\bm\theta}_{\vv p,q}(\ve)\cap \cM ) \ \le\ 2 \, n \, C \, |q|^{-1-\frac1n-\ve}\,.
\end{equation}
Furthermore, if $\frac12B^{\bm\theta}_{\vv p,q}(\ve)\cap
\cM\not=\emptyset$ then
\begin{equation}\label{e:018}
\mu( B^{\bm\theta}_{\vv p,q}(\ve)\cap \cM) \ \ge\ {\mbox{\large
$\frac12$ }}  \min\{C^{-1},|I|\}\ |q|^{-1-\frac1n-\ve}\,.
\end{equation}
\end{lemma}

\begin{proof}
Let $x,x'\in I$ be such that $\vv f(x),\vv f(x')\in
B^{\bm\theta}_{\vv p,q}(\ve)$. Then, by the definition of
$B^{\bm\theta}_{\vv p,q}(\ve)$ together with (\ref{e:015}) and
(\ref{e:016}), we have that $|qx+p_1+\theta_1|< |q|^{-1-\ve}$ and
$|qx'+p_1+\theta_1|< |q|^{-\frac1n-\ve}$. On taking the obvious
difference we find that  $|q(x-x')|<2 \, |q|^{-\frac1n-\ve}$.
Hence
\begin{equation}\label{e:019}
    |x-x'|<2 \, |q|^{-1-\frac1n-\ve}\,.
\end{equation}
Let $J$ denote the smallest interval that contains all $x$ such
that $\vv f(x)\in B^{\bm\theta}_{\vv p,q}(\ve)$. Then, in view of
(\ref{e:019}) it follows that $|J|\le2 \, |q|^{-1-\frac1n-\ve}$.
Therefore
$$
 \mu(B^{\bm\theta}_{\vv p,q}(\ve)\cap \cM) \ \le \ \int_{J}|\vv
 f'(x)|_2dx \ \stackrel{(\ref{e:016})}{\le} \ \int_Jn \, C \, dx \ \le \  n \, C \, |J| \ \le \ 2 \, n \, C \,
 |q|^{-1-\frac1n-\ve} \ ,
$$
which is precisely (\ref{e:017}).  In order to prove (\ref{e:018}),
fix as we may  a point $x\in I$ such that $\vv f(x)\in
\frac12B^{\bm\theta}_{\vv p,q}(\ve)$. Equivalently,
\begin{equation}\label{e:020}
    |q\vv f(x)+\vv p+\bm\theta|< {\mbox{\large
$\frac12$ }} |q|^{-\frac1n-\ve}\,.
\end{equation}
Now take any $x'\in I$ such that $|x-x'|<\delta \,
|q|^{-1-\frac1n-\ve}$, where $\delta:=\frac12\min\{C^{-1},|I|\}$.
By the Mean Value Theorem and (\ref{e:016}), we have that
\begin{equation}\label{e:021}
    |\vv f(x')-\vv f(x)| \ \le \  C \, \delta \,  |q|^{-1-\frac1n-\ve} \ \le \  {\mbox{\large
$\frac12$ }} \, |q|^{-1-\frac1n-\ve}\ .
\end{equation}
On combining  (\ref{e:020}) and (\ref{e:021}),  we obtain that
$$
|q\vv f(x')+\vv p+\bm\theta| \ = \ |(q\vv f(x)+\vv
p+\bm\theta)-q(\vv f(x)-\vv f(x'))| \ < \  {\mbox{\large $\frac12$
}}  |q|^{-\frac1n-\ve}+ {\mbox{\large $\frac12$ }}
|q|^{-\frac1n-\ve}\ = \  |q|^{-\frac1n-\ve}\ .
$$
Thus, for any $x'\in J':=\{x'\in I:|x-x'|<\delta \,
|q|^{-1-\frac1n-\ve}\}$ we have that $\vv f(x')\in
B^{\bm\theta}_{\vv p,q}(\ve)$. Since $\delta<|I|/2$,  we have that
$|J'|\ge \delta q^{-1-\frac1n-\ve}$. Therefore,
$$
 \mu(B^{\bm\theta}_{\vv p,q}(\ve)\cap \cM) \ \ge \ \int_{J'}|\vv
 f'(x)|_2dx \ \stackrel{f_1(x)=x}{\ge} \ \int_{J'}dx \ = \  |J'| \ \ge \  \delta
 \, |q|^{-1-\frac1n-\ve}    \  .
$$
This is precisely   (\ref{e:018}) and completes the proof of the
lemma.
\end{proof}

The following statement is an immediate consequence of the Lemma
\ref{lemma1}.

\begin{lemma}\label{lemma1+}
Let $\cM=\{\vv f(x):x\in I\}$ be a curve satisfying
$(\ref{e:015})$ and $(\ref{e:016})$. Then for any choice of $\vv
p\in\Z^n$,  $q\in\Z$,  $\ve>0$ and $\bm\theta\in\R^n$ such that
\begin{equation}\label{e:022}
    |q| \; > \; \frac2\ve
\end{equation}
and $B^{\bm\theta}_{\vv p,q}(\ve)\cap \cM\not=\emptyset$, we have
\begin{equation*}
\mu(B^{\bm\theta}_{\vv p,q}(\ve)\cap \cM) \ \le\ \frac{4 \, n \,
C}{\min\{C^{-1},|I|\}} \ \  |q|^{-\ve/2} \ \mu(B^{\bm\theta}_{\vv
p,q}(\ve/2)\cap \cM) \ .
\end{equation*}
\end{lemma}

\begin{proof}
Simply note that (\ref{e:022}) implies that $B^{\bm\theta}_{\vv
p,q}(\ve)\subset \frac12B^{\bm\theta}_{\vv p,q}(\ve/2)$ and apply
Lemma~\ref{lemma1}.
\end{proof}

\subsection{The disjoint and non-disjoint balls}

The following decomposition of $\cS_n^{\bm\theta}(\ve) \cap \cM $ represents
the key component in establishing Theorem \ref{t1}*. As we shall see, it
is extremely simple yet very  effective.

 Fix an $\ve>0$. The set $\cS_n^{\bm\theta}(\ve)  \cap \cM  $ can be written in
 the following manner to bring to the forefront its limsup nature:

\begin{equation}\label{e:024}
 \La_n^{\bm\theta}(\ve) := \bigcap_{s=1}^\infty\ \bigcup_{q\ge s}\bigcup_{\vv
 p\in\Z^n} B^{\bm\theta}_{\vv p,q}(\ve)\cap \cM\,.
\end{equation}
Since $\cM$ is bounded, for each $q$ the above union over $\vv p
\in\Z^n $ is in fact finite. We now  make a crucial distinction
between two natural types of balls appearing in (\ref{e:024}).

Fix  $\vv p \in\Z^n $ and $q\in\Z\nz$. Clearly,  there exists a unique integer $t=t(q)$  such that
$2^t\le |q|< 2^{t+1}$. The ball $B^{\bm\theta}_{\vv p,q}(\ve)$
is said to be  \emph{disjoint}\/ if for every $q'\in\Z$ with
$2^t\le |q'|<2^{t+1}$ and every $\vv p'\in\Z^n$
 \begin{equation*}
  B^{\bm\theta}_{\vv p,q}(\ve/2)\cap B^{\bm\theta}_{\vv
  p',q'}(\ve/2)\cap \cM=\emptyset\,.
 \end{equation*}
Otherwise, the ball $B^{\bm\theta}_{\vv p,q}(\ve)$ is said to be
\emph{non-disjoint}.

Naturally, the notion of disjoint and non-disjoint balls enables us to
decompose the set $ \La_n^{\bm\theta}(\ve)$ into two limsup subsets:
\begin{equation*}
\cD^{\bm\theta}_{n}(\ve) \ := \ \bigcap_{s=0}^\infty\
\bigcup_{t\ge s}\ \bigcup_{2^t\le
|q|<2^{t+1}}\bigcup_{\stackrel{\scriptstyle\vv
 p\in\Z^n}{B^{\bm\theta}_{\vv p,q}(\ve)\text{ is disjoint}}}  \!\!\!\!\!\! B^{\bm\theta}_{\vv p,q}(\ve)\cap \cM\,,
\end{equation*}
and
\begin{equation*}
\cN^{\bm\theta}_{n}(\ve) \ = \ \bigcap_{s=0}^\infty\
\bigcup_{t\ge s}\ \bigcup_{2^t\le
|q|<2^{t+1}}\bigcup_{\stackrel{\scriptstyle\vv
 p\in\Z^n}{B^{\bm\theta}_{\vv p,q}(\ve)\text{ is non-disjoint}}} \!\!\!\!\!\!\!\!\!  B^{\bm\theta}_{\vv p,q}(\ve)\cap \cM\,.
\end{equation*}

\noindent Formally,
$$
\cS_n^{\bm\theta}(\ve)  \cap \cM \  = \   \La_n^{\bm\theta}(\ve) \ =  \  \cD^{\bm\theta}_{n}(\ve) \cup \cN^{\bm\theta}_{n}(\ve)  \ .
$$

\bigskip

\subsection{The finale}

\noindent Our aim is to establish (\ref{e:013}). In view of the above decomposition of   $\La_n^{\bm\theta}(\ve) $, this will
clearly follow on showing that
\begin{equation*}\label{e:SV027}
\mu (\cN^{\bm\theta}_{n}(\ve) )   \ = \ \mu
(\cD^{\bm\theta}_{n}(\ve) ) \ = \ 0 \ .
\end{equation*}
Naturally, we deal with the disjoint and non-disjoint sets separately.

\medskip

\noindent\textbf{The disjoint case: } By the definition of
disjoint balls, for every fixed $t$ we have that
\begin{equation*}\label{e:028}
\begin{array}{rcl}
  \displaystyle\sum_{2^t\le |q|<2^{t+1}}\sum_{\stackrel{\scriptstyle\vv
 p\in\Z^n}{B^{\bm\theta}_{\vv p,q}(\ve)\text{ is disjoint}}} \!\!\!\!  \mu(B^{\bm\theta}_{\vv p,q}(\ve/2)\cap
 \cM ) & = & \displaystyle \mu \Big(\bigcup_{2^t\le |q|<2^{t+1}}\bigcup_{\stackrel{\scriptstyle\vv
 p\in\Z^n}{B^{\bm\theta}_{\vv p,q}(\ve)\text{ is disjoint}}} \!\!\!\!\!\!\!\! B^{\bm\theta}_{\vv p,q}(\ve/2)\cap
 \cM\Big) \\[8ex]
   & \le & \mu(\cM) \ < \ \infty  \ .
\end{array}
\end{equation*}
This together with Lemma~\ref{lemma1+}, implies that for $2^t > 2/\ve$
$$
  \sum_{2^t\le |q|<2^{t+1}}\sum_{\stackrel{\scriptstyle\vv
 p\in\Z^n}{B^{\bm\theta}_{\vv p,q}(\ve)\text{ is disjoint}}} \mu(B^{\bm\theta}_{\vv p,q}(\ve)\cap
 \cM) \ \ll \  2^{-t\ve/2}\, .
$$
The implied constant in the Vinogradov symbol $\ll$ does not depend
on $t$. Therefore,
\begin{equation}\label{e:029}
\sum_{t\ge s}  \  \sum_{2^t\le
|q|<2^{t+1}}\sum_{\stackrel{\scriptstyle\vv
 p\in\Z^n}{B^{\bm\theta}_{\vv p,q}(\ve)\text{ is disjoint}}} \!\!\!\!\! \mu(B^{\bm\theta}_{\vv p,q}(\ve)\cap
 \cM ) \ \ll \ \sum_{t\ge s }2^{-t\ve/2} \ \to \ 0 \quad\text{as }
 s \to\infty\,.
\end{equation}
By definition,
\begin{equation}\label{e:030}
\cD^{\bm\theta}_{ n}(\ve) \ \subset \  \bigcup_{t\ge s}\
\bigcup_{2^t\le |q|<2^{t+1}}\bigcup_{\stackrel{\scriptstyle\vv
 p\in\Z^n}{B^{\bm\theta}_{\vv p,q}(\ve)\text{ is disjoint}}} \!\!\!\! B^{\bm\theta}_{\vv p,q}(\ve)\cap \cM
\end{equation}
for arbitrary $s$ and by (\ref{e:029}) the measure of the right hand side  of
(\ref{e:030}) tends to $0$ as $s\to\infty$. Therefore, the left hand side
of (\ref{e:030}) must have zero measure; i.e.
\begin{equation*}\label{e:sv029}
\mu (\cD^{\bm\theta}_{ n}(\ve) )  \ = \  0  \ .
\end{equation*}

\bigskip

\noindent\textbf{The non-disjoint case: } Let $B^{\bm\theta}_{\vv
p,q}(\ve)$ be a non-disjoint ball and let $t=t(q)$ be  as above. Clearly
$$
    B^{\bm\theta}_{\vv p,q}(\ve)\subset B^{\bm\theta}_{\vv
    p,q}(\ve/2)\,.
$$
By the definition of non-disjoint balls, there is another ball
$B^{\bm\theta}_{\vv p',q'}(\ve/2)$ with $2^t\le |q'|<2^{t+1}$ such
that

\begin{equation}\label{e:031}
    B^{\bm\theta}_{\vv p,q}(\ve/2)\cap B^{\bm\theta}_{\vv
    p',q'}(\ve/2)\cap \cM\not=\emptyset\,.
\end{equation}

\noindent It is easily seen  that $q'\not=q$, as otherwise we would have that $B^{\bm\theta}_{\vv
p,q}(\ve/2)\cap B^{\bm\theta}_{\vv p',q}(\ve/2)=\emptyset$.

Take any point $ \vv y $ in the non-empty set appearing in (\ref{e:031}).
By the definition of $B^{\bm\theta}_{\vv p,q}(\ve/2)$ and
$B^{\bm\theta}_{\vv p',q'}(\ve/2)$, it follows that
\begin{equation*}\label{e:032}
|q\vv y +\vv p+\bm\theta| \ < \ |q|^{-1-\frac\ve2} \ \le \
2^{t(-\frac1n- \frac\ve2)}
\end{equation*}
and
\begin{equation*}\label{e:033}
|q'\vv y +\vv p'+\bm\theta| \ < \ |q'|^{-1-\frac\ve2} \ \le \
2^{t(-\frac1n-\frac\ve2)}\,.
\end{equation*}
On combining  these inequalities  in the obvious manner, we deduce that
\begin{equation}\label{e:034}
|\underbrace{(q-q')}_{q''}\vv y +\underbrace{(\vv p-\vv p')}_{\vv
p''}| \ < \ 2\cdot 2^{t(-\frac1n-\frac\ve2)}  \ < \
2^{(t+2)(-\frac1n-\frac\ve3)}
\end{equation}
for all $t$ sufficiently large. Furthermore,  $0<|q''|\le 2^{t+2}$ which together with
(\ref{e:034}) yields that
\begin{equation*}\label{ee:033}
|q''\vv y +\vv p''| \ < \ |q''|^{-\frac1n-\frac\ve3}\,.
\end{equation*}
If the latter inequality holds for infinitely many different
$q''\in\Z$,  then $\vv y \in \cS_n(\ve/3) \cap \cM$.    Otherwise,
there is  a fixed pair $(\vv p'', q'')  \in \Z^n \times \Z\nz$  such
that (\ref{e:034}) is satisfied  for infinitely many $t$. Thus, we
must have that  $q''\vv y+\vv p''=0$ and so  $\vv y$ is a rational
point.  The upshot of the non-disjoint case is that
$$
\cN^{\bm\theta}_{n}(\ve) \ \subset \  (\cS_n(\ve/3) \cap \cM )  \
\cup   \  ( \Q^n \cap \cM ) \  ,
$$
where $\Q^n$ is the set of rational points in $\R^n$. In view of
(\ref{e:014}) and since $\Q^n$ is countable, it follows that
$$
\mu(\cN^{\bm\theta}_{n}(\ve)) \ = \ 0 \ .
$$

\bigskip

\noindent This together with the analogues statement for $\cD^{\bm\theta}_{ n}(\ve)$ 
establishes
(\ref{e:013}) and thereby completes  the proof of  Theorem \ref{t1}\!*.
\newline\hspace*{\fill}$\boxtimes$

\bigskip

\noindent{\bf Acknowledgements. \ } We would like to thank the
organisers of the conference `Diophantine and Analytical Problems
in Number Theory' held at Moscow State University (Jan. 29 - Feb.
2, 2007) for giving us the opportunity to give a talk -- it was a
most enjoyable experience. Also, SV thanks the Ministry of
Internal Affairs for their kind invitation to visit Russia and
both the Royal Society and the Clay Mathematics Institute for their generous
financial support. Finally, SV must thank his young girls --
Ayesha and Iona -- who really wanted to come to Moscow to buy more
Russian dolls and play in the snow. Some other time girls --
that's a promise!

\providecommand{\bysame}{\leavevmode\hbox
to3em{\hrulefill}\thinspace}
\providecommand{\MR}{\relax\ifhmode\unskip\space\fi MR }
\providecommand{\MRhref}[2]{%
  \href{http://www.ams.org/mathscinet-getitem?mr=#1}{#2}
} \providecommand{\href}[2]{#2}

\vspace{5mm}

{ \small

\noindent Victor V. Beresnevich: Department of Mathematics,
University of York,

\vspace{-2mm}

\noindent\phantom{Victor V. Beresnevich: }Heslington, York, YO10
5DD, England.


\noindent\phantom{Victor V. Beresnevich: }e-mail: vb8@york.ac.uk

\vspace{5mm}

\noindent Sanju L. Velani: Department of Mathematics, University
of York,

\vspace{-2mm}

 ~ \hspace{17mm}  Heslington, York, YO10 5DD, England.


 ~ \hspace{17mm} e-mail: slv3@york.ac.uk

 }

\end{document}